\documentclass[a4paper,12pt]{article}
\usepackage{amsthm}
\usepackage{amsfonts}
\usepackage{amsmath}
\usepackage{amssymb}
\usepackage{diagrams}
\usepackage{hyperref}
\usepackage[pdftex]{graphicx}

\newtheorem{theorem}{Theorem}[section]

\newtheorem{proposition}[theorem]{Proposition}

\theoremstyle{definition}

\theoremstyle{remark}
\newtheorem{remark}[theorem]{Remark}

\numberwithin{equation}{section}

\newcommand{\Z}{\mathbb{Z}}

\newcommand{\zmod}[1]{\Z/#1 \Z}

\newcommand{\diag}{\operatorname{diag}}
\newcommand{\bq}{/\!\!/}

\DeclareMathOperator{\Id}{Id}

\begin{document}

\begin{center}
\textbf{The classification of compact simply connected biquotients in dimension 4 and 5}

\bigskip

Jason DeVito
\end{center}

\begin{abstract}

We classify all compact simply connected biquotients of dimension $4$ and $5$.  In particular, all pairs of groups $(G,H)$ and embeddings $H\rightarrow G\times G$ giving rise to a particular biquotient are classified.
\end{abstract}

\section{Introduction}
If $M$ is a homogeneous space with full isometry group $G$, then any subgroup $K$ of $G$ naturally acts on $M$.  In some instances, this action is effectively free, and hence, the quotient $M/K$ is a smooth manifold, called a biquotient.  Alternatively, given a compact Lie group $G$ and a homomorphism $f = (f_1,f_2):H\to G\times G$, there is an induced action of $H$ on $G$ given by $h\ast g = f_1(h)g f_2(h)^{-1}$.  When this action is effectively free, the orbit space, denoted $G\bq H$, naturally has the structure of a smooth manifold and is called a biquotient.

If $G$ is endowed with its bi-invariant metric, then the $H$ action on $G$ is by isometries, and hence induces a metric on the quotient.  By O'Niell's formulas \cite{On1}, this implies that all biquotients carry a metric of non-negative sectional curvature.  Biquotients were introduced by Gromoll and Meyer \cite{GrMe1} when they showed that for a particular embedding of $Sp(1)$ into $Sp(2)\times Sp(2)$, the biquotient $Sp(2)\bq Sp(1)$ is diffeomorphic to an exotic sphere, providing the first example of an exotic sphere with non-negative sectional curvature.  Further, until the recent example due to Grove, Verdiani, and Ziller \cite{GVZ} and Dearicott \cite{De}, all known examples of compact manifolds with positive sectional curvature were diffeomorphic to biquotients.  See \cite{Ber}, \cite{AW}, \cite{Wa}, \cite{Es1}, \cite{Ba1}, \cite{PW1}.  Furthermore, all known examples of manifolds with almost or quasi-positive curvature are diffeomorphic to biquotients.  See \cite{Wi}, \cite{PW2}, \cite{EK}, \cite{Ke1}, \cite{Ke2}, \cite{KT}, \cite{Ta1}, and \cite{D1}.

Recently, biquotients have been used in the classification of non-negatively curved manifolds of small dimension with a large symmetry group.  See \cite{Sim1}, \cite{GGK}, and \cite{GGS} for examples already using the $5$-dimensional classification appearing in this paper.

Because each description of a manifold as a biquotient gives rise to a different family of non-negatively curved metrics, it seems desirable to not only have a classification of manifolds diffeomorphic to a biquotient, but also to classify which groups give rise to a given manifold.  Totaro \cite{To1} has shown that if $M\cong G\bq H$ is a compact, simply connected biquotient, then $M$ is also diffeomorphic to $G'\bq H'$ where $G'$ is simply connected, $H'$ is connected, and no simple factor of $H'$ acts transitively on any simple factor of $G'$.  We call such biquotients reduced, and will classify only the reduced ones.

In dimension 2 and 3, it is easy to see that only $S^2$ and $S^3$ arise and that their only description as reduced biquotients is as $S^2\cong SU(2)/S^1$ and $S^3 \cong SU(2)/\{e\}$.

The goal of this paper is to prove the analogous results for compact simply connected 4 and 5 dimensional biquotient.  A follow up paper will contain the classification results for $6$ and $7$ dimensional simply connected biquotients.  We have

\begin{theorem}\label{ThmA}Suppose $M^4\cong G\bq H$ is a reduced compact simply connected biquotient and that the $H$ action on $G$ is not homogeneous.  Then $G$, $H$, and the image of $H$ in $G\times G$ appear in the following table.

\begin{center}

\begin{tabular}{|c|c|c|c|}

\hline

$M$ & $G$ & $H$ & $H\rightarrow G\times G$\\

\hline

$S^4$ & $Sp(2)$ & $Sp(1)^2$ & $Sp(1)\times \Delta Sp(1)$\\

$S^4$ & $SU(4)$ & $SU(3)\times SU(2)$ & $SU(3)\times \Delta SU(2)$ \\

$S^4$ & $Spin(8)$ & $Spin(7)\times SU(2)$ & $Spin(7)\times \Delta SU(2)$ \\

$S^4$ & $Spin(8)$ & $Spin(7)\times SU(2)$ & $\overline{Spin(7)}\times \overline{SU(2)}$ \\

$S^4$ & $Spin(7)$ & $G_2\times SU(2)$ & $G_2\times \overline{SU(2)}$\\

\hline

$\mathbb{C}P^2$ & $SU(3)$ & $SU(2)\times S^1$ & $\diag(zA,\overline{z}^2)\times \diag(z^4, z^4, \overline{z}^8)$ \\ 

$\mathbb{C}P^2$ & $SU(4)$ & $Sp(2) \times S^1$ & $Sp(2)\times \diag(z,z,z,\overline{z}^3)$\\

\hline

$S^2\times S^2$ & $Sp(1)^2$ & $T^2$ &  $(z^2, wz^n)\times(1,z^n)$ $n$ even \\

$\mathbb{C}P^2\sharp -\mathbb{C}P^2$ & $Sp(1)^2$ & $T^2$ & $(z^2, wz^n)\times (1,z^n)$ $n$ odd \\

$\mathbb{C}P^2\sharp \mathbb{C}P^2$ & $Sp(1)^2$ & $T^2$ & $(zw,zw^2)\times(w,z)$\\

\hline

\end{tabular}

\end{center}

\end{theorem}

Here, $\overline{Spin(7)}$ denotes the image of the spin representation.  The subgroup $\overline{SU(2)}$, which is isomorphic to $SU(2)$, is the inverse image of the standard block embedding of $SO(3)$ into either $SO(7)$ or $SO(8)$ under the natural projection from $Spin(7)$ or $Spin(8)$.  This notation $-\mathbb{C}P^2$ indicates $\mathbb{C}P^2$ with the opposite orientation and $M$ is given up to diffeomorphism.

Rational homotopy theory implies that if $M^4$ is diffeomorphic to a biquotient, then it must be homeomorphic to one of the $5$ spaces listed above.  It follows from our results that $M$ is, in fact, diffeomorphic to one of those $5$ spaces.  We remark that Totaro \cite{To1} has already shown all $5$ diffeomorphism types can be written as biquotients, where only the cases of $\mathbb{C}P^2\pm \sharp \mathbb{C}P^2$ were not previously known.

In dimension $5$, we prove

\begin{theorem}\label{ThmB}

Suppose $M^5\cong G\bq H$ is a reduced compact simply connected biquotient and that the $H$ action on $G$ is not homogeneous.  Then $G$, $H$, and the image of $H$ in $G\times G$ appear in the following table.

\begin{center}

\begin{tabular}{|c|c|c|c|}

\hline

$M$ & $G$ & $H$ & $H\rightarrow  G\times G$\\

\hline

$S^3\times S^2$ & $Sp(1)^2$ & $S^1$ & $\diag(z^a, z^b)\times \diag(z^c, z^d)_1$ \\

$S^3\,\hat{\times}\,S^2$ & $Sp(1)^2$ & $S^1$ & $\diag(z^a,z^b)\times \diag(z^c,z^d)_2$ \\

\hline

\end{tabular}

\end{center}

\end{theorem}

In both cases, we have $\gcd(a,b,c,d) = 1$.  The subscript $1$ indicates $\gcd(a^2-c^2, b^2-d^2)= 1$ while the subscript $2$ indicates $\gcd(a^2-c^2, b^2-d^2) = 4$.  The notation $S^3\,\hat{\times} \, S^2$ indicates the unique nontrivial $S^3$ bundle over $S^2$.    The manifold $M$ is given up to diffeomorphism.

In \cite{Pa1}, Pavlov gives a classification of $5$-dimensional biquotients up to diffeomorphism, but he does not classify all ways of expressing a given manifold as a biquotient.  We also point out a slight error Pavlov's formula for showing the manifold $S^3\,\hat{\times}\, S^2$ is a biquotient.

The outline of the paper is as follows.  In section 2, we will first cover some basic facts about biquotients.  Then, using previous work and some simple rational homotopy theory, we reduce Theorems A and B to classifying the effectively free actions and diffeomorphism types of biquotients of the form $(SU(2)\times SU(2))\bq T^2$ and of the form $(SU(2)\times SU(2))\bq S^1$.  Section 3 carries this out in the case of $(SU(2)\times SU(2))\bq T^2$ and section 4 carries this out in the case $(SU(2)\times SU(2))\bq S^1$.

This paper is a portion of the author's Ph.D. thesis and he is greatly indebted to Wolfgang Ziller for helpful discussions and guidance.

\section{Preliminaries and Reductions}

A homomorphism $f = (f_1, f_2):H\rightarrow G\times G$, which we will always assume has finite kernel, defines an action of $H$ on $G$ by $h\ast g = f_1(h)g f_2(h)^{-1}$.

An action is called effectively free if whenever any $h\in H$ fixes any point of $G$ then it fixes all points of $G$.  It is called free if the only element which fixes any point is the identity.  One easily sees that a biquotient action of $H$ on $G$ is effectively free iff whenever $f_1(h)$ is conjugate to $f_2(h)$ in $G$, then $f_1(h) = f_2(h)\in Z(G)$.  Likewise, a biquotient action of $H$ on $G$ is free iff whenever $f_1(h)$ is conjugate to $f_2(h)$ in $G$, then $f_1(h) = f_2(h) = e\in G$.

In order to reduce the scope of the classification, we will use the following fact:

\begin{proposition}\label{diffeostabilize} Suppose $f:H\rightarrow G\times G$ induces an effectively free action.  Then, after any of the following modifications of $f$, the new induced action is effectively free and the quotients are naturally diffeomorphic.

1.  For any automorphism $f'$ of $H$, change $f$ to $f\circ f'$

2.  For any element $g=(g_1,g_2)\in G\times G$, change $f$ to $C_g\circ f$, where $C_g$ denotes conjugation.

3.  For any automorphism $f'$ of $G$, change $f$ to $(f',f')\circ f$.

\end{proposition}

\begin{remark}One may think that the $(f',f')$ in $3.$ can be replaced by $(f'_1,f'_2)$ for any pair of automorphisms of $G$.  However, this is not the case.  For example, if $G = Sp(1)\times Sp(1)$ and $H = S^1$ with the embedding $z\rightarrow \big( (z,1), (1,z)\big)$, then the induced action is free.  On the other hand, if $f'_1 = \Id$ while $f'_2$ interchanges the two $S^3$ factors, then the action induced by $(f'_1,f'_2)\circ f$ is not even effectively free - every element of $S^1$ fixes infinitely many points.

\end{remark}

We will only classify biquotients and the corresponding actions up to these three modifications.

Further, Totaro \cite{To1} has proven that if $M$ is compact, simply connected, and diffeomorphic to a biquotient, then $M$ is diffeomorphic to a biquotient $G\bq H$ where $G$ is compact, simply connected, and semisimple, $H$ is connected, and no simple factor of $H$ acts transitively on any simple factor of $G$.  We will henceforth assume all biquotients to be in this form.

\bigskip

One of the main tools involved in the classification of biquotients is rational homotopy theory.  A manifold $M$ is said to be rationally elliptic if $\dim \pi_\ast(M)\otimes\mathbb{Q} < \infty$.  All Lie group are known to be rationally elliptic with all even rational homotopy groups trivial.  Further, given any fiber bundle $F\rightarrow E\rightarrow B$, if two of the spaces are rationally elliptic, so is the third by the long exact sequence in rational homotopy groups.  Since any biquotient $G\bq H$ with ineffective kernel $H'$ gives rise to a principal $H/H'$-bundle $H/H'\rightarrow G\rightarrow G\bq H$, it follows that all biquotients are rationally elliptic.

It turns out, the topology of a simply connected rationally elliptic manifold is very constrained.  In \cite{Pa2}, Pavlov proves

\begin{proposition}(Pavlov)  Suppose $M$ is a compact simply connected rationally elliptic manifold.  If $M$ is $4$-dimensional, $M$ has the same rational homotopy type as either $S^4$, $\mathbb{C}P^2$, $S^2\times S^2$, or $\mathbb{C}P^2\#\mathbb{C}P^2$.  If $M$ is $5$-dimensional, $M$ has the same rational homotopy type as $S^5$ or $S^3\times S^2$.

\end{proposition}

Since the cohomology rings with coefficients in $\mathbb{Q}$ of $S^4$, $\mathbb{C}P^2$, and $S^5$ are all generated by a single element, the same follows for any $M$ which has the same rational homotopy type.  In particular, we can use the classification of Kapovitch and Ziller \cite{KZ} and independently Totaro \cite{To1} which shows that, in fact, $M$ must be diffeomorphic to one of $S^4$, $\mathbb{C}P^2$, $S^5$, or $SU(3)/SO(3)$.  Further, if such an $M$ is diffeomorphic to a reduced biquotient $G\bq H$, then $G$ must be simple.

In Eschenburg's Habilitation \cite{Es2}, summarized in \cite{Zi}, he classifies all maximal biquotient actions of maximal rank on simple groups.  It follows from his classification that our lists in Theorem A and B is complete in the case that $M$ has the rational homotopy type of a compact rank one symmetric space.  Hence, in order to prove Theorem A and Theorem B, it remains to check the cases where $M$ has the rational homotopy type of $S^2\times S^2$, $\mathbb{C}P^2\#\mathbb{C}P^2$ or $S^3\times S^2$.

\begin{proposition} Suppose $M$ is compact and simply connected and that it has the rational homotopy type of $S^2\times S^2$, $\mathbb{C}P^2\#\mathbb{C}P^2$, or $S^3\times S^2$.  If $M\cong G\bq H$ is a reduced biquotient, then $G = SU(2)\times SU(2)$ and $H = S^1$ or $S^1\times S^1$.

\end{proposition}

\begin{proof}

Suppose $M\cong G\bq H$ has the rational homotopy type of $S^2\times S^2$ or $\mathbb{C}P^2\#\mathbb{C}P^2$.  From the long exact sequence or rational homotopy groups associated the fibration $$H\rightarrow G\rightarrow M$$ we see that $\pi_1(H)\otimes\mathbb{Q}\cong \pi_2(M)\otimes\mathbb{Q}\cong\mathbb{Q}^2$.  It follows that $H$, up to cover, must split as $H = H'\times S^1\times S^1$.  Considering the circle bundle $$S^1\rightarrow G\bq (H'\times S^1\times\{e\}) \rightarrow G\bq (H'\times S^1\times S^1)$$ and the associated long exact sequence of rational homotopy groups, we see $G\bq (H'\times S^1)$ is a reduced compact simply connected $5$-dimensional manifold with the same rational homotopy type as $S^3\times S^2$.  But Pavlov \cite{Pa1} shows that this implies $G = SU(2)\times SU(2)$ while $H' = \{e\}$.

\end{proof}

\section{\texorpdfstring{$(SU(2)\times SU(2))\bq T^2$}{SU(2) x SU(2) // T2}}

First, we will classify all effectively free actions of $T^2$ on $SU(2)\times SU(2)$.  Note that by Proposition \ref{diffeostabilize} we may assume that the image $f(T^2)$ lies in the maximal torus of $(SU(2)\times SU(2))^2$.  Working directly with homomorphisms from $T^2$ into $(SU(2)\times SU(2))^2$ is cumbersome, so instead, we will adopt a more geometric description of the action.  For concreteness, we take $S^1 = \{z\in\mathbb{C}: |z|=1\}$ and $S^3 = \{(p,q)\in \mathbb{C}^2: |p|^2 + |q|^2 = 1\}$.

It is easy to see that the action of $T^2$ on $SU(2)$ given by $$ (z,w)\ast U = \diag \left(z^{A} w^{B}, \overline{z}^A \overline{w}^B\right)U \diag\left(z^C w^D,  \overline{z}^C \overline{w}^D\right)^{-1}$$ is equivalent to the action of $T^2$ on $S^3$ given by $$(u,v)\ast(p,q) = (z^{A-B}w^{C-D} p, z^{A+B}w^{C+D}q)$$ under the isomorphism $SU(2)\rightarrow S^3$ mapping $\begin{bmatrix} a & b\\ -\overline{b} & \overline{a}\end{bmatrix}$ to $(a,b)$.

Conversely, the action of $T^2$ on $S^3$ given by $$(z,w)\ast (p,q) = (z^{2a} w^{2b} p, z^{2c} w^{2d} q) $$ is equivalent to one on $SU(2)$  with $A = a+c$, $B = b+d$, $C = a-c$, and $D = b-d$.  It follows that any linear action of $T^2$ on $S^3\times S^3$ is orbit equivalent to a biquotient action of $T^2$ on $SU(2)\times SU(2)$.

Our first step towards classifying the effectively free actions is to simplify their description.

\begin{proposition}Consider the action $$(z,w)\ast \big((p_1,q_1), (p_2,q_2)\big) = \big( (z^a w^b p_1, z^c w^d q_1), (z^e w^f p_2, z^g w^h q_2)   \big)$$ with $\gcd(a,c,e,g) = \gcd(b,d,f,h) = 1$.  Assume the action is effectively free.  Then there is a change of coordinates on $T^2$ for which the action has the form $$(z,w)\ast \big( (p_1,q_1), (p_2,q_2)\big) = \big( (z p_1, z^\alpha w^\beta q_1), (w p_2, z^\gamma w^\delta q_2) \big).$$

\end{proposition}

\begin{proof} Set $D = \det\begin{bmatrix} a&e\\b&f\end{bmatrix}$.  If $D=0$ then there is nontrivial integral solution $(n_1,n_2)$ to the simultaneous equations $$\left\{\begin{aligned} an_1 &+ en_2  &= 0 \\ b n_1 &+ f n_2 &= 0\end{aligned}\right.$$  Then points of the form $(z^{n_1}, z^{n_2})$ all fix $\big( (1,0), (1,0)\big)$, so either the ineffective kernel is infinite or the action is not free.  Hence, $D\neq 0$.

In order to find new coordinates, we interpret $(a,c,e,g)$ and $(b,d,f,h)$ as vectors in $\mathbb{R}^4$ and consider their span.  Using the fact that $D\neq 0$, it is easy to see there is another basis of the form $(\nu,\alpha,0,\gamma)$ and $(0,\beta,\kappa,\delta)$ consisting solely of integers.  In the new basis $(z,w)$, the action now looks like $$(z,w)\ast \big( (p_1,q_1), (p_2,q_2)\big) = \big( (z^{\nu} p_1, z^\alpha w^\beta q_1), (w^{\kappa} p_2, z^\gamma w^\delta q_2) \big)$$ and we may assume without loss of generality that $$\gcd(\nu,\alpha, \gamma) = \gcd(\kappa,\beta,\delta) = 1.$$

If $z$ is a $\nu$th root of unity, then $(z,1)$ fixes $\big((1,0), (1,0)\big)$.  Since the action is effectively free, this implies $(z,1)$ fixes every point which in turn implies that $z$ is an $\alpha$th and $\gamma$th root of unity.  So, every $\nu$th root of unity is a $\gcd(\nu,\alpha,\gamma) = 1$st root of unity, so $\nu=1$.  An analogous argument shows $\kappa = 1$ as well.

\end{proof}

Using Proposition \ref{diffeostabilize} we can replace $z$, $w$, $p_i$ or $q_i$ by their complex conjugates, and hence we can assume without loss of generality that $\alpha, \beta$, and $\delta$ are all bigger than or equal to $0$. Also, with this description, it is clear that an action is effectively free iff it is free.  Further, the argument in first paragraph of the previous proposition implies $\alpha\delta -\beta\gamma \neq 0$.

\begin{proposition}\label{4free}  The action $$(z,w)\ast \big( (p_1,q_1), (p_2,q_2)\big) = \big( (z p_1, z^\alpha w^\beta q_1), (wp_2, z^\gamma w^\delta q_2) \big)$$ is free iff $\alpha = \delta = 1$ and $|1-\beta \gamma| = 1$.

\end{proposition}

\begin{proof}

The key observation is that if an element $(z,w)$ fixes any point in $S^3\times S^3$, this it must fix a point of the form $\big( (1,0), (1,0)\big)$, $\big( (1,0), (0,1)\big)$, $\big( (0,1),(1,0)\big)$, or $\big( (0,1), (0,1)\big)$.  So, to guarantee the action is free, it is enough to show every $(z,w)$ moves each of these $4$ points.  Note that the point $\big( (1,0), (1,0)\big)$ is automatically moved by all nontrivial $(z,w)$.

The point $\big( (1,0), (0,1) \big)$ is fixed by $(z,w)$ iff $w$ is a $\delta$th root of unity.  So the action always moves this point iff $\delta = 1$.  An analogous argument shows the point $\big( (0,1),(1,0)\big)$ is always moved iff $\alpha = 1$.

Finally, the element $(z,w)$ fixes the point $\big( (0,1), (0,1)\big)$ iff $zw^\beta = 1$ and $z^\gamma w = 1$.  Raising the second equation to the $\beta$th power, we see that any solution of this must have $z$ equal to a $(1-\beta\gamma)$th root of unity.  Raising the first equation to the $\gamma$th power shows the same of $w$.  Then, it is easy to see that the solutions are precisely the pairs of the form $(\overline{w}^\beta, w)$ where $w$ is any $(1-\beta\gamma)$th root of unity.  Thus, in order to guarantee that we have a free action, we must have $|1-\beta\gamma|=1$.

\end{proof}

From here, it is obvious that, up to interchanging $z$ and $w$, either $\gamma = 0$ and $\beta$ is arbitrary or $\beta = 1$ and $\gamma = 2$.  It remains to compute the diffeomorphism type of the quotients.

\begin{proposition} The action with $\gamma = 0$ and $\beta$ arbitrary has quotient $S^2\times S^2 $ iff $\beta$ is even and quotient $\mathbb{C}P^2\sharp -\mathbb{C}P^2$ iff $\beta$ is odd.  For the exceptional action, the quotient is diffeomorphic to $\mathbb{C}P^2\sharp \mathbb{C}P^2$.

\end{proposition}

\begin{proof}
In \cite{To1}, Totaro shows that the exceptional action has quotient diffeomorphic to $\mathbb{C}P^2\sharp\mathbb{C}P^2$, so we only prove the theorem in the case where $\gamma = 0$.

Since $\gamma = 0$, we see that the first factor $S^1$ of $T^2$ acts only on the first factor $S^3$ via the Hopf action with quotient $S^2$.  So, $(S^3\times S^3)\bq (S^1\times \{e\}) \cong S^2\times S^3$.  The projection onto the second factor $(S^2\times S^3)/S^1 \rightarrow S^3/S^1\cong S^2$ gives $(S^3\times S^3)\bq T^2$ the structure of an $S^2$ bundle over $S^2$.  By a standard clutching function argument, there are precisely two such bundles up to isomorphism, one of which is $S^2\times S^2$ and the other of which is $\mathbb{C}P^2\sharp -\mathbb{C}P^2$.

Geometrically, the $\{e\}\times S^1$ action on $S^2\times S^3$ is by rotating the $S^2$ factor $|\gamma|$ times while acting via the Hopf action on the $S^3$ factor.  Since rotating $S^2$ $|\gamma|$ times is homotopically trivial (in $SO(3)$) iff $|\gamma|$ is even, the quotient is $S^2\times S^2$ iff $|\gamma|$ is even.

\end{proof}

This completes the proof of Theorem \ref{ThmA}.

\section{\texorpdfstring{$(SU(2)\times SU(2))\bq S^1$}{SU(2) x SU(2) // S1}}

We begin by classifying the free actions.  Up to conjugation, a general biquotient action of $S^1$ on $SU(2)\times SU(2)$ has the form $$z\ast(A,B) = \left(\begin{bmatrix} z^a & \\ & \overline{z}^a\end{bmatrix} A \begin{bmatrix} z^c & \\ & \overline{z}^c\end{bmatrix}^{-1}, \begin{bmatrix} z^b & \\ & \overline{z}^b\end{bmatrix}B \begin{bmatrix} z^d & \\ & \overline{z}^d\end{bmatrix}^{-1}  \right)$$ where we may assume without loss of generality that $\gcd(a,b,c,d)=1$.

\begin{proposition}\label{5free} Such an action is effectively free iff for any choice of signs, we have $\gcd(a\pm c, b\pm d) = 1$ or $2$.

\end{proposition}

\begin{remark}
Since the parities of $a+c$ and $a-c$ are the same, an effectively free biquotient action either has all four $\gcd$s equal to $1$ or all four equal to $2$.  It is easy to see that all $4$ $\gcd$s are $1$ iff $\gcd(a^2-c^2, b^2-d^2) = 1$ and all $4$ $\gcd$s are $2$ iff $\gcd(a^2-c^2, b^2-d^2) = 4$.
\end{remark}

\begin{proof}Since two matrices are conjugate in $SU(2)$ iff they have the same eigenvalues, we see that $\begin{bmatrix} z^a & \\ & \overline{z}^a\end{bmatrix}$ and $\begin{bmatrix} z^c & \\ & \overline{z}^c\end{bmatrix}$ are conjugate iff $z^a = z^c$ or $z^a = \overline{z}^c$.  So, we see that a given element $z\in S^1$ fixes a point in $SU(2)\times SU(2)$ iff $z^{a\pm c} = 1 = z^{b\pm d}$ for some choice of signs.  In order to the action to be effectively free, a necessary and sufficient condition is that whenever $z^{a\pm c} = 1 = z^{b\pm d}$, then $$z^a = z^c = \pm 1 \in Z(SU(2))$$ and likewise $$z^b=z^d = \pm 1\in Z(SU(2)).$$  Equivalently, the action is effectively free iff $$\gcd(a\pm c, b\pm d) | \gcd(2a,2b,2c,2d) = 2.$$

\end{proof}

We will call a $4$-tuple of integers $(a,b,c,d)$ admissible if $$ \gcd(a^2-c^2, b^2-d^2) = 1 \text{ or } 4.$$  Pavlov \cite{Pa1} has shown that, given any biquotient $(SU(2)\times SU(2))\bq S^1$, the quotient is diffeomorphic to either $S^3\times S^2$ or $S^3\,\hat{\times}\,S^2$.  His main tool is the Barden-Smale classification of $5$-dimensional manifolds, proven by Smale \cite{Sm1} in the spin case and Barden \cite{Ba1} in the non-spin case.  Their result implies that a compact simply connected $5$-dimensional manifold with $H_2\cong \mathbb{Z}$ is diffeomorphic to either $S^3\times S^2$ or $S^3\, \hat{\times}\, S^2$ and it is easy to see that any biquotient of the form $(SU(2)\times SU(2))\bq S^1$ has $H_2 \cong \mathbb{Z}$.

While both $S^3\times S^2$ and $S^3\,\hat{\times}\, S^2$ have the same cohomology ring, the second Stiefel-Whitney class distinguishes them.  This class is trivial for $S^3\times S^2$ and nontrivial for $S^3\,\hat{\times}\, S^2$.

We will follow a method of Singhof \cite{Si1} for computing the characteristic classes of these biquotients.  Because we are interested in the Stiefel-Whitney classes, all of the following homology and cohomology groups are assumed to have $\zmod{2}$ coefficients.  In order to apply this method, we first recall some background.  To begin with, given any compact Lie group $G$, we will let $EG$ denote a contractible space on which $G$ acts freely and $BG = EG/G$ will be the classifying space of $G$.  If the $H$ biquotient action on $G$ is free, the projection $\pi:G\rightarrow G\bq H$ is an $H$-principal bundle, hence is classified by a map $\phi_H:G\bq H\rightarrow BH$.

Eschenburg \cite{Es3} has shown

\begin{proposition}\label{commute} Suppose $\phi:H\rightarrow G\times G$ induces a free action of $H$ on $G$ and consider the fibration $\sigma:G\rightarrow B\Delta G\rightarrow BG\times BG$ induced by the natural inclusion $\Delta G\rightarrow G\times G$.  There is a map $\phi_G:G\bq H\rightarrow B\Delta G$ so that the following is, up to homotopy, a pullback of fibrations.

\begin{diagram}
G &  & G \\
\dTo & & \dTo \\
G \bq H  & \rTo^{\phi_G} & B{\Delta G} \\
\dTo^{\phi_H} && \dTo^{B\sigma} \\
BH & \rTo^{Bf} & BG\times BG\\
\end{diagram}

\end{proposition}

In order to compute the Stiefel-Whitney classes of a biquotient, we use the notions of $2$-groups and $2$-roots of a Lie group.  A $2$-group of a compact Lie group $G$ is any subgroup isomorphic to $(\zmod{2})^n$ for some $n$.  The $2$-roots of $G$ are defined analogously to the roots:  the adjoint representation of $G$, when restricted to a maximal torus $T_G$, breaks into root spaces.  Likewise when the adjoint representation of $G$ is restricted to a maximal $2$-group $Q_G$ it breaks into $2$-root spaces.  We can view each $2$-root as a map $Q_G\rightarrow \mathbb{Z}/2\mathbb{Z}$ which induces, via the fibration $Q_G\rightarrow EQ_G\rightarrow BQ_G$, a map $H_1(BQ_G)\rightarrow \mathbb{Z}/2\mathbb{Z}$, that is, as an element of $H^1(BQ_G)$.  More generally, given a basis for $Q_G$ as a $\zmod{2}$-vector space, the dual basis can be canonically identified as generators of $H^1(BQ_G)$.  Using this identification, Singhof has shown \cite{Si1}

\begin{theorem}(Singhof)\label{charclass}

Suppose $H\subseteq G\times G$ defines a free biquotient action, then the total Stiefel-Whitney class of the tangent bundle of $G\bq H$ is given as $$w(G\bq H) = \phi_G^\ast\big(\Pi_{\lambda \in \Delta^2 G}(1+\lambda)\big)\phi_H^\ast\big(\Pi_{\rho\in\Delta^2 H}(1+\rho)\big)^{-1}$$ where $\Delta^2 G$ denotes the $2$-roots of $G$ and where $\phi_G^\ast$ and $\phi_H^\ast$ are the maps induced on cohomology with $\zmod{2}$ coefficients.

\end{theorem}

The $2$-roots of the classical groups are listed in \cite{BH1}.  Note that $H=S^1$ has no nontrivial $2$ roots, nor does $SU(2)\times SU(2)$.  On the other hand, the maximal $2$-group of $U(2)\times U(2)$ is given, up to conjugacy, by all pairs of diagonal matrices with entries $\pm 1$.  If $\lambda_1$ is dual to $(\diag(-1,1), I)$, $\lambda_2$ is dual to $(\diag(1,-1),I)$, with $\mu_1$ and $\mu_2$ dual to $(I, \diag(-1,1))$ and $(I,\diag(1,-1))$ respectively, then the nontrivial $2$-roots of $U(2)$ are $\lambda_1-\lambda_2$ and $\mu_1-\mu_2$ each with multiplicity $2$.

The case where $\gcd(a,b,c,d)$ is admissible with $\gcd(a^2-c^2, b^2-d^2)=1$ describes a free action, so we can directly apply Singhof's theorem.  Since neither $H$ nor $G$ have $2$-roots, the products in Theorem \ref{charclass} are all trivial, so their pullbacks are as well.  Hence, if $\gcd(a^2-c^2, b^2-d^2) = 1$, then the biquotient is diffeomorphic to $S^3\times S^2$.

Unfortunately, if $(a,b,c,d)$ is admissible with $\gcd(a^2-c^2,b^2-d^2) = 4$, then the induced action is merely effectively free, so Singhof's approach will not directly work.  In order to circumvent this, we'll define a free biquotient action on $U(2)\times U(2)$ which preserves $SU(2)\times SU(2)$ and which has the same orbits through these points.  Singhof's approach will apply to this setting and then we can transfer this to information about $(SU(2)\times SU(2))\bq H$ by considering the following commutative diagram where $G = SU(2)\times SU(2)$, $G' = U(2)\times U(2)$, and $H=S^1$:

\begin{diagram}
G &  & G' & & G' \\
\dTo & & \dTo & & \dTo \\
G \bq H  & \rTo^{j} & G'\bq H &\rTo^{\phi_G} & B\Delta G' \\
\dTo^{\phi_H} & & \dTo^{\phi_H'} & & \dTo^{B\sigma} \\
BH & \rTo^{=} & BH & \rTo^{Bf} & BG'\times BG'\\
\end{diagram}
The two vertical fibrations on the right come from Proposition \ref{commute}, the first vertical fibration comes from the free $H$ action on $G$, and where $j:G\bq H\rightarrow G'\bq H$ is induced from the natural inclusion $G\rightarrow G'$.

The action of $H$ on $G'$ is given by the next proposition.  First note that if $\gcd(a^2-c^2,b^2-d^2) = 4$, then both $a \pm c$ and $b\pm d$ are even.

\begin{proposition}  Consider any admissible $4$-tuple $(a,b,c,d)$ with the property that $\gcd(a^2-c^2, b^2-d^2) = 4$.  Then the action of $S^1$ on $U(2)\times U(2)$ given by $$ z*(A,B) = \left(\begin{bmatrix} z^{a} & \\ & 1\end{bmatrix} A \begin{bmatrix}z^{\frac{a+c}{2}} & \\ & z^{\frac{a-c}{2}}\end{bmatrix}^{-1}, \begin{bmatrix} z^{b} & \\ & 1\end{bmatrix} B  \begin{bmatrix} z^{\frac{b+d}{2}} & \\ & z^{\frac{b-d}{2}}\end{bmatrix}^{-1} \right) $$  is free and the orbits through points in $SU(2)\times SU(2)$ are the same as in the action induced by the $4$-tuple.

\end{proposition}

\begin{proof}
Observe that if an effectively free biquotient action is determined by a homomorphism $f:H\rightarrow G\times G$ and $f':H\rightarrow \Delta Z(G) \subseteq G\times G$ is any homomorphism, then the action determined by the map $F:H\rightarrow G\times G$ with $F(h) = f(h)f'(h)$ is also effectively free with the same orbits.  In this case, we extend the action determined by the admissible $4$-tuple $(a,b,c,d)$ to one on $G'$ and then use $f':H\rightarrow \Delta Z(G')$ given by $$z\mapsto \big( \diag(z^a,z^a), \diag(z^b,z^b), \diag(z^a,z^a), \diag(z^b,z^b)\big).$$  We see that while $-1$ is not in the kernel of $f$, it is in the kernel of $ff'$.  Dividing $S^1$ by $\langle -1\rangle$ gives the above action.

The action is now free because the image of $S^1$ intersects $\Delta Z(G)$ iff $z$ is a $\gcd\left(a, \frac{a+c}{2}, b, \frac{b+d}{2}\right)$th root of unity, that is, iff $z=1$.

\end{proof}

We can also easily relate the topology of $G\bq H$ to that of $G'\bq H$.  First recall that the map from $SU(n)\times S^1$ to $U(n)$ given by sending $(A,z)$ to $\diag(z,1,...,1) A$ is a diffeomorphism.  Further, if an action of $H$ on $G'$ is invariant under the determinant $\det:G'\rightarrow T^2$, then this action is equivalent via this diffeomorphism to an action of $H$ on $G\times T^2$ which acts trivially on the $T^2$ factor.  If follows that $G'/H\cong G /H \times T^2$.  In particular, the map $j$ in the above commutative diagram maps $G\bq H$ to $G\bq H \times \{(1,1)\}\in G'\bq H$.

We can now apply Singhof's method to these actions.  By Theorem \ref{charclass}, since $H$ has no nontrivial $2$-roots, our goal is to determine whether or not $$j^\ast\phi_G^\ast \left[(1+\lambda_1 + \lambda_2)(1+\mu_1 + \mu_2) \right]^2 = 0$$ when restricted to $H^2(G\bq H)$.  Since $B\sigma^\ast: H^\ast(B(G\times G))\rightarrow H^\ast(B\Delta G)$ maps $(\lambda_1 + \lambda_2)\otimes 1$ to $\lambda_1 + \lambda_2$, this is equivalent to computing $$j^\ast\phi_G^\ast B\sigma^\ast\left[\big(1+(\lambda_1+\lambda_2)\otimes 1\big)\big(1+(\mu_1+\mu_2)\otimes 1\big)\right]^2$$ when restricted to $H^2(G\bq H)$.  Since $B\sigma\, \phi_G\, j = Bf\, \phi_H'\,j = Bf\,\phi_H$, this is equivalent to computing $$\phi_H^\ast Bf^\ast\left[ \big(1+(\lambda_1+\lambda_2)\otimes 1\big)\big(1+(\mu_1+\mu_2)\otimes 1\big)\right]^2$$ in $H^2(G\bq H)$.

By the Gysin sequence associated to the fiber bundle $$S^1\rightarrow SU(2)\times SU(2)\rightarrow (SU(2)\times SU(2))\bq S^1,$$ we see that the Euler class of the bundle generates $H^2((SU(2)\times SU(2))\bq S^1)$.  In particular, the classifying map $\phi_H:(SU(2)\times SU(2))\bq S^1\rightarrow BS^1$ induces an isomorphism on $H^2$, so is injective.  Thus, we need only determine whether or not $$Bf^\ast \left[ \big(1+(\lambda_1 + \lambda_2)\otimes 1\big)\big(1+(\mu_1 + \mu_2)\otimes 1\big)\right]^2$$ is trivial or not in $H^2(BS^1)$.

In order to compute $Bf^\ast$, we use a technique developed by Borel \cite{Bo1}.  If $Q_G\subseteq G$ is a maximal $2$-group, then Borel proves that the induced map $H^\ast(BG)\rightarrow H^\ast(BQ_G)$ is injective and he computes in the image for each of the classical groups.  For $H=S^1$, the maximal $2$-group $Q_{H}$ is generated by the element $-1\in H$.  If $w$ is dual to this element, then we can identify $H^\ast(BQ_{H}) \cong \zmod{2}[w]$.  In this notation, Borel proves the image of $H^\ast(BH)$ is then $\zmod{2}[w^2]$.

For $G'= U(2)\times U(2)$, as above, we take $Q_{G'}$ to be the collection of all diagonal matrices with entries $\pm 1$.  We also take the dual basis mentioned above, $\{\lambda_1, \, \lambda_2, \, \mu_1, \, \mu_2\}$.  This allows us to identify $$H^\ast(B_{G'})\cong \zmod{2}[\lambda_1, \, \lambda_2, \, \mu_1, \, \mu_2].$$  From here, the K\"{u}nneth formula gives $$H^\ast(B(Q_{G'}\times Q_{G'}))\cong\zmod{2}[\lambda_1, \, \lambda_2, \, \mu_1, \, \mu_2]\otimes \zmod{2}[\lambda_1, \, \lambda_2, \, \mu_1, \, \mu_2].$$  Since the induced map $H^\ast(B(G'\times G'))\rightarrow H^\ast(B(Q_{G'}\times Q_{G'}))$ is injective, and since $f(Q_h)\subseteq Q_{G'}\times Q_{G'}$, we can compute $Bf^\ast$ by computing it as a map $$H^\ast(B(Q_{G'}\times Q_{G'}))\rightarrow H^\ast (BQ_H).$$

In this description, notice that $Bf_2^\ast$ is trivial when restricted to elements of the form $x\otimes 1$, so, we need only compute $$Bf_1^\ast \left[\big( 1+(\lambda_1 + \lambda_2)\otimes 1\big) \big(1+(\mu_1 + \mu_2)\otimes 1\big)\right]^2$$ in $H^2(BH)$.

Concretely, we have $f(z) = (f_1(z), f_2(z))$ with $$f_1(z) = \big( \diag\left(z^a, 1\right), \diag\left( z^b, 1\right)\big)$$ and $$f_2(z) = \big(\diag\left(z^{\frac{a+c}{2}}, z^{\frac{a-c}{2}}\right), \diag\left(z^{\frac{b+d}{2}}, z^{\frac{b-d}{2}}\right)\big).$$  Thus, we see  $$f(-1) = \big( \diag((-1)^a,1), \diag((-1)^b,1)\big).$$   Now we note that $a$ and $b$ must have opposite parity.  For, if $a$ and $b$ are both even, then since $\gcd(a\pm c, b\pm d) = 2$, we must have $c$ and $d$ even as well.  This contradicts $\gcd(a,b,c,d)=1$.  Similarly if $a$ and $b$ are both odd, then $c$ and $d$ must both be odd.  Then, by checking cases mod $4$, we see that $4|(a\pm c, b\pm d)$ for some choice of signs, again giving a contradiction.  We will thus assume without loss of generality that $a$ is odd and $b$ is even, so $f_1(-1) = \big( \diag(-1,1), \diag(1,1)\big)$.  It follows that $f_1^\ast(\lambda_1) = w$ while $f_1^\ast(\lambda_2) = f_1^\ast(\mu_1) = f_1^\ast(\mu_2) = 0$.

From here, a simple calculation shows 

\begin{proposition}With the setup as above, $Bf_1^\ast(\lambda_1\otimes 1) = w$ while $Bf_1^\ast$ is $0$ on the other generators of $H^\ast(B(Q_{G'}\times Q_{G'}))$.  Thus, we have $$Bf_1^\ast\left( 1+(\lambda_1+\lambda_2)\otimes 1)(1+(\mu_1+\mu_2)\otimes 1\right)^2 = w^2\in H^\ast(B(S^1))$$ so is nontrivial.  In particular, the biquotient $G\bq H$ has nontrivial second Stiefel-Whitney class, so is diffeomorphic to $S^3\,\hat{\times}\,S^2$.

\end{proposition}

This completes the proof of Theorem \ref{ThmB}.

\bibliographystyle{plain}
\bibliography{../bibliography}

\begin{thebibliography}{10}

\bibitem{AW}
S.~Aloff and N.~Wallach.
\newblock An infinite family of $7$-manifolds admitting positively curved
  {R}iemannian structures.
\newblock {\em Bull. Amer. MAth. Soc.}, 81:93--97, 1975.

\bibitem{Ba1}
D.~Barden.
\newblock Simply connected five-manifolds.
\newblock {\em Ann. Math.}, 82:365--385, 1965.

\bibitem{Ber}
M.~Berger.
\newblock Les vari\'{e}t\'{e}s {R}iemanniennes homog\'{e}nes normales
  simplement connexes \'{a} courbure strictement positive.
\newblock {\em Ann. Scuola Norm. Sup. Pisa}, 15:179--246, 1961.

\bibitem{Bo1}
A.~Borel.
\newblock La cohomologie mod 2 de certains espaces.
\newblock {\em Comm. Math. Helv.}, 27:165--197, 1953.

\bibitem{BH1}
A.~Borel and F.~Hirzebruch.
\newblock Characteristic classes and homogeneous spaces i.
\newblock {\em Amer. J. of Math.}, 80:458--538, 1958.

\bibitem{De}
O.~Dearricott.
\newblock A 7-manifold with positive curvature.
\newblock {\em Duke Math. J.}, 158:307--346, 2011.

\bibitem{D1}
J.~DeVito.
\newblock The classification of simply connected biquotients of dimension 7 or
  less and 3 new examples of almost positively curved manifolds.
\newblock {\em Thesis, University of Pennsylvania}.

\bibitem{Es1}
J.~Eschenburg.
\newblock New examples of manifolds with strictly positive curvature.
\newblock {\em Invent. Math.}, 66:469--480, 1982.

\bibitem{Es2}
J.~Eschenburg.
\newblock Freie isometrische aktionen auf kompakten {L}ie-gruppen mit positiv
  gekr$\ddot{\text{u}}$mmten orbitr$\ddot{\text{a}}$umen.
\newblock {\em Schriften der Math. Universit$\ddot{\text{a}}$t
  M$\ddot{\text{u}}$nster}, 32, 1984.

\bibitem{Es3}
J.~Eschenburg.
\newblock Cohomology of biquotients.
\newblock {\em Manu. Math.}, 75:151--166, 1992.

\bibitem{EK}
J.~Eschenburg and M.~Kerin.
\newblock Almost positive curvature on the gromoll-meyer 7-sphere.
\newblock {\em Proc. Amer. Math. Soc.}, 136:3263--3270, 2008.

\bibitem{GGK}
F.~Galaz-Garcia and M.~Kerin.
\newblock Cohomogenity-two torus actions on non-negatively curved manifolds of
  low dimension.
\newblock {\em ar{X}iv:1111.1640}, 2011.

\bibitem{GGS}
F.~Galaz-Garcia and C.~Searle.
\newblock Low-dimensional manifolds with non-negative curvature and maximal
  symmetry rank.
\newblock {\em Proc. Amer. Math. Soc.}, 139:2559--2564, 2011.

\bibitem{GrMe1}
D.~Gromoll and W.~Meyer.
\newblock An exotic sphere with nonnegative sectional curvature.
\newblock {\em Ann. Math.}, 100:401--406, 1974.

\bibitem{GVZ}
K.~Grove, L.~Verdiani, and W.~Ziller.
\newblock An exotic ${T^1 S^4}$ with positive curvature.
\newblock {\em Geom. Funct. Anal.}, 21:499--524, 2011.

\bibitem{KZ}
V.~Kapovitch and W.Ziller.
\newblock Biquotients with singly generated rational cohomology.
\newblock {\em Geom. Dedicata}, 104:149--160, 2004.

\bibitem{Ke2}
M.~Kerin.
\newblock Some new examples with almost positive curvature.
\newblock {\em Geometry and Topology}, 15:217--260, 2011.

\bibitem{Ke1}
M.~Kerin.
\newblock On the curvature of biquotients.
\newblock {\em Math. Ann.}, 352:155--178, 2012.

\bibitem{KT}
M.~Kerr and K.~Tapp.
\newblock A note on quasi-positive curvature conditions.
\newblock {\em ar{X}iv:1211.3897}, 2012.

\bibitem{On1}
B.~O'Neill.
\newblock The fundamental equations of a submersion.
\newblock {\em Michigan Math. J.}, 13:459--469, 1966.

\bibitem{Pa2}
A.V. Pavlov.
\newblock Estimates for the {B}etti numbers of rationally elliptic spaces.
\newblock {\em Siberian Math. J.}, 43:1080--1085, 2002.

\bibitem{Pa1}
A.V. Pavlov.
\newblock Five dimensional biquotients of {L}ie groups.
\newblock {\em Siberian Math. J.}, 45:1080--1083, 2004.

\bibitem{PW2}
P.~Petersen and F.~Wilhelm.
\newblock Examples of {R}iemannian manifolds with positive curvature almost
  everywhere.
\newblock {\em Geom. and Top.}, 3:331--367, 1999.

\bibitem{PW1}
P.~Petersen and F.~Wilhelm.
\newblock An exotic sphere with positive sectional curvature.
\newblock {\em ar{X}iv:0805.0812}, 2008.

\bibitem{Sim1}
F.~Simas.
\newblock Nonnegatively curved 5-manifolds with non-abelian symmetry.
\newblock {\em ar{X}iv:1212.5022}, 2012.

\bibitem{Si1}
W.~Singhof.
\newblock On the topology of double coset manifolds.
\newblock {\em Math. Ann.}, 297:133--146, 1993.

\bibitem{Sm1}
S.~Smale.
\newblock On the structure of 5-manifolds.
\newblock {\em Ann. Math.}, 75:38--46, 1962.

\bibitem{Ta1}
K.~Tapp.
\newblock Quasi-positive curvature on homogeneous bundles.
\newblock {\em J. Diff. Geo.}, 65:273--287, 2003.

\bibitem{To1}
B.~Totaro.
\newblock Cheeger manifolds and the classification of biquotients.
\newblock {\em J. Diff. Geo.}, 61:397--451, 2002.

\bibitem{Wa}
N.~Wallach.
\newblock Compact homogeneous riemannian manifolds with strictly positive
  curvature.
\newblock {\em Ann. of Math.}, 96:277--295, 1972.

\bibitem{Wi}
B.~Wilking.
\newblock Manifolds with positive sectional curvature almost everywhere.
\newblock {\em Invent. Math.}, 148:117--141, 2002.

\bibitem{Zi}
W.~Ziller.
\newblock On {E}schenburg's habilitation on biquotients.
\newblock {\em ar{X}iv:0909.0167}, 2009.

\end{thebibliography}

\end{document}